\documentclass[a4paper]{article}
\usepackage[utf8]{inputenc}
\usepackage[]{amsmath}
\usepackage{amsfonts,amssymb}
\usepackage[russian,english]{babel}
\usepackage[disable]{todonotes}
\usepackage{graphicx}
\usepackage{amsthm} 
\usepackage[unicode]{hyperref}
\usepackage{verbatim}

\newtheorem{theorem}{Theorem}
\newtheorem{conjecture}{Conjecture}
\newtheorem{lemma}{Lemma}

\newtheorem{ttheorem}{Theorem} 

\newcounter{step}
\newcommand{\step}[1]{\refstepcounter{step}\textbf{Step \thestep.} \emph{#1.}\ }

\newtheorem{proposition}{Proposition}
\theoremstyle{definition}
\newtheorem{definition}{Definition}
\newtheorem{remark}{Remark}
\newcommand{\wt}{\widetilde}
\newcommand{\eps}{\varepsilon}
\newcommand{\rarc}[2]{[#1,#2 \rangle}
\newcommand{\larc}[2]{\langle #1,#2]}
\newcommand{\rPeab}[2]{P_\eps^{\rarc{#1}{#2}}}
\newcommand{\lPeab}[2]{P_\eps^{\larc{#2}{#1}}}
\newcommand{\Dep}{D_\eps^+}
\newcommand{\Dem}{D_\eps^-}
\newcommand{\Depm}{D_\eps^\pm}

\newcommand{\padi}[2]{\frac{\partial #1}{\partial #2}}
\newcommand{\xgc}{x^{gc}}
\title{Duck factory on the two-torus: multiple canard cycles without geometric constraints}
\author{Ilya Schurov\thanks{Higher School of Economics}\ \thanks{The present paper uses results obtained under the support of the project No 12-01-0227 ``Dynamics of Josephson junction and slow-fast systems on the two-torus'' within the program ``The HSE scientific foundation'' in 2013/2014}\ \ and Nikita Solodovnikov\footnotemark[1]\ \thanks{Supported by project MK-7567.2013.1}\ \thanks{The article was prepared within the framework of the Academic Fund Program at the National Research University Higher School of Economics (HSE) in 2016-2017 (grant No 16-05-0066) and supported within the framework of a subsidy granted to the HSE by the Government of the Russian Federation for the implementation of the Global Competitiveness Program.}}
\date{}

\begin{document}
\maketitle
\let\subsectionautorefname\sectionautorefname
\begin{abstract}
Slow-fast systems on the two-torus are studied. As it was shown before, canard cycles are generic in such systems, which is in drastic contrast with the planar case. It is known that if the rotation number of the Poincar\'e map is integer and the slow curve is connected, the number of canard limit cycles is bounded from above by the number of fold points of the slow curve. In the present paper it is proved that there are no such geometric constraints for non-integer rotation numbers: it is possible to construct a generic system with ``as simple as possible'' slow curve and arbitrary many limit cycles.
\end{abstract}
\sloppy
\section{Introduction}
Consider a generic slow-fast system on the two-dimensional torus
\begin{equation}\label{eq:slow_fast_system}
    \left\{
    \begin{aligned}
        \dot{x}&=f(x,y,\eps)\\
        \dot{y}&=\eps g(x,y,\eps)\\
    \end{aligned}
    \right. 
    \hspace{40pt}
    (x,y)\in\mathbb{T}^2\cong \mathbb{R}^2/(2\pi\mathbb{Z}^2),\quad \eps\in(\mathbb R_+,0)
\end{equation}
Assume that $f$ and $g$ are smooth enough and $g>0$. The dynamics of this system is guided by the \emph{slow curve}:
$$M=\{(x,y)\mid f(x,y,0)=0\}.$$
It consists of equilibrium points of the fast motion (i.e. the motion determined by system~\eqref{eq:slow_fast_system} for $\eps=0$). Particularly, one can consider two parts of the slow curve: one is stable (consists of attracting hyperbolic equilibrium points) and the other is unstable (consists of repelling hyperbolic equlibrium points). On the plane $\mathbb R^2$, there is rather simple description of the generic trajectory of~\eqref{eq:slow_fast_system}: it consists of interchanging phases of the slow motion along the stable parts of the slow curve and the fast jumps along the straight lines $y=const$ near the folds of the slow curve~\cite{MR}. On the two-torus, more complicated behaviour can be locally generic.

\begin{definition}
    A solution (or trajectory) is called \emph{canard} if it contains an arc of length bounded away from zero uniformly in $\eps$ that keeps close to the unstable part of the slow curve and simultaniously contains an arc (also of length bounded away from zero uniformly in $\eps$) that keeps close to the stable part of the slow curve.
\end{definition}
This definition is a bit informal, more rigorous one will be given in \autoref{section:mainResults} (see \autoref{def:canard}). Canards are not generic on the plane: one have to introduce an additional parameter to get an attracting canard cycle. (See e.g.~\cite{KS}.) However, they are generic on the two-torus, as was conjectured in~\cite{GI} and proved in~\cite{IS1}. 

Let us explain this phenomena briefly. Assume that there exists a global cross-section $\Gamma=\{y=const\}$ transversal to the field. Then one can define the Poincar\'e map $P_\eps\colon\Gamma\to\Gamma$. It is a diffeomorphism of a circle. 
The rotation number $\rho(\eps)$ of the map $P_\eps$ 
continuously depends on $\eps$. 
For generic system \eqref{eq:slow_fast_system} function $\rho(\eps)$ is a Cantor function (also known as devil's staircase) whose horizontal steps occur at rational values and (in the general case) correspond to the existence of hyperbolic periodic points of the map $P_\eps$. These in turn correspond to limit cycles of the original vector field. More precisely, if the Poincar\'e map has a rotation number with a denominator $n$ then the initial vector field has a limit cycle which makes $n$ full passes along the slow direction of the torus $y$. In particular, fixed points of the Poincar\'e map correspond to limit cycles which make only one pass along the slow direction of the torus. 

While hyperbolic limit cycles present, the rotation number is preserved under small perturbations. So when the rotation number increases, the limit cycles have to bifurcate through saddle-node (parabolic) bifurcation. Near the critical value of the parameter, the derivative of the Poincare map for both colliding cycles has to be close to 1. This is possible only if the cycles spend comparable time near the stable and the unstable parts of the slow curve, and thus they are canards.

The next natural question is to provide an estimate for the number of canard cycles that can born in a generic slow-fast system on the two-torus. The answer to this question for the case of integer rotation number and a rather wide class of systems was given in~\cite{IS2}.

\begin{theorem}\label{thm:old}
For generic slow-fast system on the two-torus with contractible nondegenerate connected slow curve the number of limit cycles that make one pass along the axis of the slow motion is bounded by the number of fold points of the slow curve. This estimate is sharp in some open set in the space of slow-fast systems on the two-torus.
\end{theorem}
In the present paper we consider the case of non-integer rotation number that is not covered by \autoref{thm:old}. We also conjecture that our arguments can be applied to systems with unconnected slow curves (see \autoref{sec:conjectures}). The latter case is of special interest because slow-fast systems with unconnected slow curve appear naturally in physical applications, e.g. in the modelling of circuits with Josephson junction~\cite{KRS}.

Our main result states that in contrast with~\autoref{thm:old} for non-integer rotation number there are \emph{no geometric constraints} on the number of (canard) limit cycles.
Particularly, for any desired odd number of limit cycles $l$ we construct open set in the space of the slow-fast systems on the two-torus with \emph{convex} slow curve (i.e. having only two fold points) with exactly $l$ canard cycles that make two passes along the axis of the slow motion. (The corresponding Poincar\'e map has half-integer rotation number.) See \autoref{thm:main}.

\section{Main results}\label{section:mainResults}
In this section we state our main results. We are interested only in the phase curves of system~\eqref{eq:slow_fast_system}, so one can divide it by $g$ and consider without loss of generality case of $g\equiv 1$.

\begin{definition}
    A slow curve $M$ is called \emph{simple} if it is smooth and connected,  its lift to the covering coordinate plane is contained
		in the interior of the fundamental square $\{|x|<\pi,\ |y|<\pi\}$ and is
		convex. 
\end{definition}
We only consider simple slow curves. This, in particular, implies that there are
two jump points (forward and backward jumps), which are the far right and the
far left points of $M$ (see \autoref{fig:general-view}; here and below we 
assume that the fast coordinate $x$ is vertical and the slow coordinate $y$ is
horizontal). We denote them by $G^-$ and $G^+$ respectively.  We assume without
loss of generality that $G^\pm=(0,\mp 1)$. (Here and below every equation with
$\pm$'s and $\mp$'s corresponds to a couple of equations: with all top and all
bottom signs.) This can be achieved by an appropriate change of coordinates
respecting fibration $\{y=const\}$. 

\begin{definition}\label{def:nondeg}
    A simple slow curve $M$ is called \emph{nondegenerate} if and only if
    \begin{enumerate}
\item The following nondegenericity assumption holds in every point $(x,y)\in
	M\setminus\{G^+,G^-\}$:
	\begin{equation}\label{eq-nondeg-nondeg}
		\padi{f(x,y,0)}{x}\ne 0.
	\end{equation}

	\item \label{enum-cond-last} The following nondegenericity assumptions hold
		in the jump points: 
		\begin{equation}\label{eq-nondeg-main}
			\left.\padi{^2 f(x,y,0)}{x^2}\right|_{G^\pm} \ne 0,\quad
			\left.\padi{f(x,y,0)}{y}\right|_{G^\pm} \ne 0
		\end{equation}
    \item \label{enum-cond-int} Let $M^-$ and $M^+$ be the stable and the unstable parts of the slow curve
        respectively. Then
        \begin{equation} \label{eq:nondeg-int}
            \int_{M^+} f'_x(x, y, 0) dy + \int_{M^-} f'_x(x, y, 0) dy \ne 0.
        \end{equation}
\end{enumerate}
\end{definition}

\begin{definition}\label{def:canard}
    Fix some small $\delta>0$ (to be chosen later). Denote by $I_\delta$ the segment $I_\delta=[-1+\delta, 1-\delta]=:[\wt \alpha^-,\wt \alpha^+]$. Denote also the segments 
    $$\Sigma^-=\{(0,y)\mid y\in I_\delta\},$$
    and
    $$\Sigma^+=\{(\pi,y)\mid y\in I_\delta\}.$$
    Every trajectory that cross either of these two segments will be called \emph{canard}.
\end{definition}
This definition of canard differs from those used
in~\cite{GI,IS1,IS2}. Our definition is stricter than the latter one and is more
classical: the trajectory that pass almost no time near
the stable part of the slow curve is not canard according to our definition. 

Now we are ready to state the main result.

\begin{ttheorem}\label{thm:main}
For every desired odd number of limit cycles $l\in 2\mathbb N+1$ there exists an open set in the space of slow-fast systems on the two-torus with the following properties. 
\begin{enumerate}
\item The slow curve $M$ is simple and nondegenerate. 
\item For every system from this set there exists a sequence of intervals $\{R_n\}_{n=0}^\infty\subset \{\eps>0\}$, accumulating at zero, such that for every $\eps\in R_n$ there exist exactly $l$ canard limit cycles. 
\end{enumerate}
\end{ttheorem}
\begin{remark}
    One can construct the desired example for any prescribed simple
    nondegenerate slow curve. Moreover, the requirement of convexity can be
    easily replaced with the less restrictive requirement: $M$ has only two fold points. This can be achieved by some smooth coordinate change preserving fibration $\{y=const\}$.
\end{remark}
\begin{remark}
No upper estimate on the number of non-canard limit cycles is given. At least
one non-canard cycle exists in our settings.
\end{remark}
The paper is organized as follows. In \autoref{sec:lc} we discuss the settings and prove the main result modulo technical Lemmas. In \autoref{sec:techn} we prove these Lemmas. In \autoref{sec:conjectures} we state a conjecture on
slow-fast systems on the two-torus with unconnected slow curves.

\section{Proof of the main result}\label{sec:lc}
In this section we provide the proof of the main result. We investigate the ``vertical'' Poincar\'e map $Q_\eps$ from the segment
$\Sigma^-$ to itself (see \autoref{def:canard}; the complete settings are given in
\autoref{section:notation}). When some special requirements satisfied, this map can be
approximated by some simple function. We provide necessary preliminary results in \autoref{section:preliminary} and describe the approximation in \autoref{ssec:heur} as Technical Lemmas~\ref{prop:Q-decomp} and \ref{prop:Q-decomp-deriv}. Then we give the proof of \autoref{thm:main} (see \autoref{section:outline}) modulo Technical Lemmas.

\subsection{Settings and notations}\label{section:notation}
Let $M$ be a slow curve. It consists of the stable ($M^-$) and unstable ($M^+$) parts and two jump points: the forward jump point~$G^-$ and the backward jump point~$G^+$, see~\autoref{fig:general-view}: 
$$M=M^+\sqcup \{G^+\}\sqcup M^-\sqcup\{G^-\}.$$
\begin{figure}[htb]
    \begin{center}
        \includegraphics[scale=0.9]{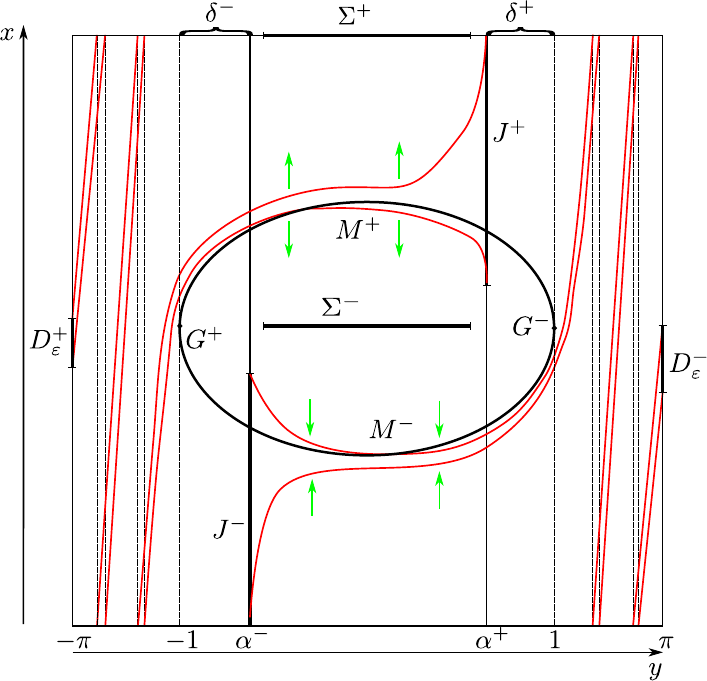}
        \caption{The slow curve and the jump points. Note that the horizontal axis is $y$ and the vertical axis is $x$}\label{fig:general-view}
    \end{center}
\end{figure}%
Recall that  $G^\pm=(0,\mp 1)$. 
We will also use the notation $M^\pm(y)$ assuming that~$M^\pm$ here are such functions that graphs $x=M^\pm(y)$ define the unstable and stable parts of the slow curve. 

Call $\Pi=S^1\times I_\delta$ a \emph{basic strip}, where $I_{\delta}$ is the same as in \autoref{def:canard}.

Fix a vertical segment $J^+$ (resp., $J^-$) that intersects $M^+$ ($M^-$) close
enough to the jump point $G^-$ ($G^+$) and does not intersect $M^-$ ($M^+$). Let 
$$y(J^\pm)=:\alpha^\pm=\pm 1\mp\delta^\pm,$$
where $\delta^\pm$ are small and $\delta^\pm<\delta$ (here $\delta$ is the same as in~\autoref{def:canard}). Note that the definition of $J^\pm$ differs from one in~\cite{GI,IS1}: instead of placing $J^+$ near $G^+$ we place it near $G^-$ and do opposite with $J^-$.

We have to reproduce the notation on oriented arcs on a circle and Poincar\'e maps from~\cite{IS1}. Consider arbitrary points~$a$ and~$b$ on the oriented circle~$S^1$. They split
the circle into two arcs. Denote the arc from point $a$ to point $b$ (in the
sense of the orientation of the circle) by~$\rarc{a}{b}$. The orientation of this
arc is induced by the orientation of the circle. Also denote the same arc with
the reversed orientation by~$\larc{a}{b}$ (see~\autoref{fig:arcs}).

\begin{figure}[htb]
    \begin{center}
        \includegraphics[scale=0.8]{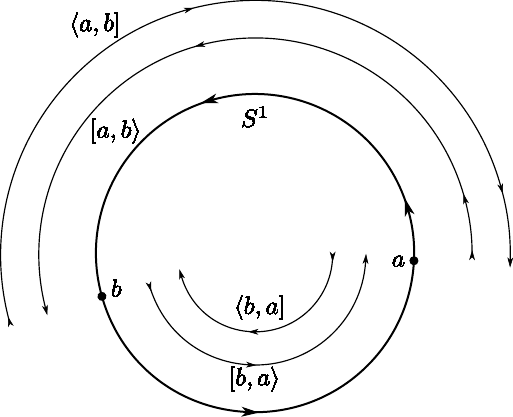}
        \caption{Orientation of the arcs}\label{fig:arcs}
    \end{center}
\end{figure}

Denote also the Poincar\'e map along the phase curves of the main system~\eqref{eq:slow_fast_system}
from the cross-section~$y=a$ to the cross-section~$y=b$ in the forward time by $\rPeab{a}{b}$. Also, let
$\lPeab{b}{a}=(\rPeab{a}{b})^{-1}$: this is the Poincar\'e map from the
cross-section $y=b$ to the cross-section $y=a$ in the backward time. This fact is
stressed by the notation: the direction of the angle bracket shows the time
direction.
\subsection{Preliminary results}\label{section:preliminary}
Denote
$$\Dep:=\lPeab{\alpha^+}{-\pi}(J^+),\quad \Dem:=\rPeab{\alpha^-}{\pi}(J^-)$$
It is proved in~\cite{GI,IS1} (see the Shape Lemmas there) that $|\Depm|=O(e^{-C/\eps})$. Note that as $\eps$ decreases to 0, $\Dep$ moves downward and $\Dem$ moves upward, making infinitely many rotations (see the Monotonicity Lemmas in~\cite{GI,IS1}) and meet each other infinitely many times. The values of $\eps$ for which $\Dep$ and $\Dem$ have nonempty intersection forms intervals $R_n$:
$$\{R_n\}_{n=0}^\infty=\{\eps>0\colon \Dep\cap\Dem\ne \varnothing\}.$$
As it was shown in~\cite{GI,IS1}, intervals $R_n$ have exponentially small length and accumulate at zero. If one pick any sequence $\eps_n\in R_n$, $n=1,2,\ldots$, then
$$\eps_n=O\left(\frac{1}{n}\right).$$
Fix some $\eps\in R_n$ and pick some point $q\in \Dep\cap\Dem$. Consider the trajectory through~$q$. In the forward time,  this trajectory makes several (about $O(1/\eps)$) rotations, then performs backward jump, follows the unstable part of the slow curve $M^+$ and finally intersects $J^+$. In the backward time, this trajectory (again after several rotations) passes near the stable part of the slow curve $M^-$ and finally intersects~$J^-$. We will call this trajectory \emph{grand canard} despite the fact that this is not a canard according to~\autoref{def:canard}.

Let $U$ be a segment of the stable or unstable part of the slow curve $M$. Consider the integral:
\begin{equation}\label{eq:Phi}
\int_U f'_{x} dy.
\end{equation}
This integral describes the expansion (if positive) or contraction (if
negative) accumulated while passing near the corresponding arc of the slow curve. Formally, the following theorem holds:
\begin{theorem}[See \cite{GI}]\label{thm:P'est}
Let $U=[A,B]\subset M^\pm$ and $X=[x_1,x_2]\times \{y(A)\}$ be a segment that contains $A$ and does not cross $M$ in points different from $A$. Then
$$
\left.\log \left(\rPeab{y(A)}{y(B)}(x)\right)'_x\right|_{X}
=
\frac{1}{\eps}\left(\int_U f'_{x} dy+O(\eps)\right).
$$
\end{theorem}
Moreover, similar (but a little weaker) estimate holds for trajectories extended through the jump point even after they make $O(1/\eps)$ rotations along the $x$-axis after the jump. The exact statement follows.
\begin{theorem}[See \cite{IS1,IS2}]\label{thm:P'est-full}
Let $U=[A,G^-]\subset M^-\cup\{G^-\}$, $X$ as in previous Theorem and $y_1$ is a point outside of the projection of $M$ to $y$-axis, such that there are no other points of that projection on the arc $\rarc{1}{y_1}$. Then the following holds:
$$
\left.\log \left(\rPeab{y(A)}{y_1}(x)\right)'_x\right|_{X}
=
\frac{1}{\varepsilon}\left(\int_U f'_{x} dy+O(\eps^\nu)\right),
$$
where $\nu\in (0,1/4]$.
\end{theorem}
Reverting the time, one can obtain a similar result for $M^+$ and $G^+$.

\subsection{Approximation of the Poincar\'e map}\label{ssec:heur}
Below we use the following notation:
\begin{equation}\label{eq:def-of-lambda}
\lambda^{\pm}(y) := f'_x(M^{\pm}(y),y).
\end{equation}

\begin{definition}\label{def:def-of-beta}
    For any $y\in [-1,1]$ let \emph{singular release point} (denoted by $\beta(y)$) be the unique root of the equation
\begin{equation}\label{eq:def-of-beta}
\int_{y}^{1}\lambda^{-}dy + \int_{-1}^{\beta(y)}\lambda^{+}dy =0
\end{equation}
\end{definition}
Due to nondegenericity assumptions \eqref{eq-nondeg-nondeg}  the singular
release point is well-defined for any $y$ such that $|\int_{y}^{1}\lambda^{-}dy|
< \int_{-1}^{1}\lambda^{+}dy$.

\begin{lemma}\label{prop:Q-decomp}
    Let $\eps\in R_n$ for some $n$ and therefore the grand canard exists.
Then the Poincar\'e map $Q_\eps\colon \Sigma^-\to \Sigma^-$ can be decomposed in
the following way:
$$
Q_{\varepsilon}\colon
\Sigma^-\stackrel{Q^-_\eps}{\to}\Sigma^+\stackrel{Q^+_\eps}{\to}\Sigma^-,$$
and the maps $Q^+_\eps$ and $Q^-_\eps$ both have the following asymptotics:
\begin{equation}\label{eq:QandBeta}
    Q^\pm_\eps(y)=\beta(y)+O(\eps^\nu)
\end{equation}
for some $\nu\in (0,1/4)$, $\beta\in C^2$. The Poincar\'e map $Q_\eps$ is defined at the point~$y$ for small $\eps$ if and only if $\beta(y)$ and $\beta(\beta(y))$ are
defined and contained in the interior of $I_\delta$.

\end{lemma}
\begin{lemma}\label{prop:Q-decomp-deriv}
Under the assumptions of \autoref{prop:Q-decomp} the following expansion holds:
\begin{equation}
    \dfrac{d}{dy}Q^\pm_\eps(y)=\dfrac{d}{dy}\beta(y)+o(1).
\end{equation}
\end{lemma}
We give the heuristic proof of \autoref{prop:Q-decomp} below in this section. We give rigorous proofs of \autoref{prop:Q-decomp} and \autoref{prop:Q-decomp-deriv} in \autoref{ssec:approx} and \autoref{ssec:Q-deriv} respectively.
\begin{proof}[Heuristic proof of \autoref{prop:Q-decomp}]
    Consider the map $Q^-_\eps$ (the map $Q^+_\eps$ can be considered in the same way). Let~$p=(0,y_0)\in \Sigma^-$ and therefore $p$ is bounded away from $M$. Consider a trajectory $\psi$ through $p$. 
    In the forward time, it attracts to $M^-$ after time $O(1)$ (``falls'') moving in the negative direction (``downwards''), see~\autoref{fig:general-view}.  
    After the fall, the trajectory follows~$M^-$ exponentially close to the grand canard (being ``above'' the grand canard) until it reaches the jump point. 
    It follows from \autoref{thm:P'est-full} that after the jump the distance between the trajectory $\psi$ and the grand canard is exponentially small and its $\log$ is approximately 
$\eps^{-1}\left(\int_{y_0}^{1}\lambda^- dy\right) < 0$.  

After the jump, the trajectory follows the grand canard during the rotation phase, then performs backward jump and passes near some segment of the unstable part of the slow curve $M^+$. 
It is possible that the trajectory will be released from the grand canard  (and thus $M^+$) at some point and then attracted to $M^-$ again before leaving the basic strip. This release will be made in the positive direction (``upward''), because the trajectory $\psi$ is above the grand canard. In this case the trajectory intersects $\Sigma^+$ and therefore $Q^-_\eps(y_0)$ is defined.

We will show that the release is possible only near $\beta(y_0)$. Indeed, the release occurs at the point where the contraction rate accumulated
during the passage near the stable part of the slow curve is compensated by the
expansion accumulated near the unstable part of the slow curve. The latter is
exponentially large and its $\log$ is approximately 
$\eps^{-1}\left(\int_{-1}^{\beta(y_0)}\lambda^+ dy\right) > 0$

Relation~\eqref{eq:def-of-beta} says that $\beta(y_0)$ is the point for which the contraction is compensated by the expansion. Thus the value of the Poincar\'e map $Q^-_\eps(y_0)$ is approximately equal to $\beta(y_0)$. The error of all calculations is of order $O(\eps^\nu)$ as in~\autoref{thm:P'est-full}. 

Rigorous proof will be given in \autoref{ssec:approx}.
\end{proof}

\subsection{Proof modulo approximation Lemmas}\label{section:outline}
\autoref{prop:Q-decomp} and \autoref{prop:Q-decomp-deriv} imply that for sufficiently small $\varepsilon$ the following holds:
\begin{equation}
	Q_{\varepsilon} = \beta\circ\beta + O(\varepsilon^{\nu}).
\end{equation}
\begin{equation}
	Q_{\varepsilon}' = (\beta\circ\beta)' + o(1).
\end{equation}
Fixed points of $Q_{\varepsilon}$ correspond to $2$-periodic points of the map
$\beta$ (including fixed points).
 
Without loss of generality consider the case 
$\left|\int_{-1}^{1}\lambda^+dy\right| > \left|\int_{-1}^{1}\lambda^-dy\right|$. The opposite case can be reduced to this one by time reverse; equality is impossible due to nondegenericity assumption~\eqref{eq:nondeg-int}. Now the map $\beta$ is an orientation reversing monotonic map:
\begin{equation}
	\beta\colon[-1,1]\to[-1,\varkappa],\quad \varkappa<1.
\end{equation}
Such a map has a unique fixed point, and all the other periodic orbits have period $2$. Thus there are only periodic points of period $2$ whose number is automatically even, and one fixed point. This explains why the number of canards is odd. 

In \autoref{lemma:any-beta} below we show that in a sense ``any'' map $\beta$
can be realized by an appropriate choice of system~\eqref{eq:slow_fast_system}:
boundary conditions near the ends of the segment $[-1, 1]$ are the only
restriction imposed on $\beta$. Such conditions cannot affect the number of
periodic orbits of $\beta$ as no periodic orbits can exist near the ends of the
segment (this is due to $\varkappa < 1$). It follows immediately that one can
obtain $\beta$ with any given odd number of hyperbolic periodic orbits (one
fixed point and several 2-periodic points). These periodic orbits
correspond to hyperbolic fixed points of the square map $\beta \circ \beta$. The
Poincar\'e map $Q_{\eps}$ is $C^1$-close to $\beta\circ \beta$. Therefore for
$\eps$ small enough $Q_{\eps}$ has the same number of fixed points as
$\beta\circ\beta$. These points correspond to hyperbolic limit cycles, which
have to be canards as they intersect $\Sigma^\pm$. They are preserved under small
perturbations of the system due to hyperbolicity. This finishes the proof of Main
Result modulo \autoref{lemma:any-beta}. The rest of the section is devoted to this
Lemma.

To give a precise statement we have to introduce several functional classes.

Consider a set $\Theta$ of smooth functions $\theta\colon[-1,1]\to \mathbb{R}$ with the following properties:
\begin{enumerate}
\item $\theta(y)>0$ for any $y\in(-1,1)$. 
\item $\theta(\pm 1)=0$.
\item $\theta'(\pm 1) \ne 0$.
\end{enumerate}
Consider a set $B_{\varkappa}$ of functions $\tilde{\beta}$ with the following properties:
\begin{enumerate}
\item $\tilde{\beta}$ is smooth and strictly decreasing on $[-1,1]$.
\item $\tilde{\beta}(1)=-1$, $\tilde{\beta}'(1)\ne 0$.
\item $\tilde{\beta}(-1) = \varkappa$.
\item There exists $\theta \in \Theta$ such that for any $y$ in the small semi-neighbourhood $V=[-1, -1+\delta_1)$ the following holds:
\begin{equation}\label{eq:beta-construction}
	\tilde{\beta}'(y) = -\sqrt{\theta(y)}
\end{equation}  
\end{enumerate}

Call a pair of function $\lambda^+$ and $\lambda^-$ \emph{admissible} provided
that the following holds:

\begin{enumerate}
	\item For any $k=0,1,\ldots$
	\begin{equation}\label{eq:l-1}
		\dfrac{d^k(\lambda^+)^{-1}}{dx^k}(\pm1)=
		\dfrac{d^k(\lambda^-)^{-1}}{dx^k}(\pm1).
	\end{equation}
	\item For some $\theta^{\pm} \in \Theta$
			\begin{equation}\label{eq:l-2}
				\lambda^+=\sqrt{\theta^+},~\lambda^-=-\sqrt{\theta^-}.
			\end{equation}
\end{enumerate}
Due to the smoothness of system \eqref{eq:slow_fast_system} and the nondegenericity assumptions (see \autoref{def:nondeg})
functions $\lambda^{\pm}$ defined by equation \eqref{eq:def-of-lambda} are
admissible. They also define function $\beta$ of class $B_{\varkappa}$.

Now we go in the opposite direction: we will show that any $\beta$ of class $B_{\varkappa}$ can be realized by admissible $\lambda^+$ and $\lambda^-$. Due to Whitney Theorem any admissible $\lambda^\pm$ can be realized by some system~\eqref{eq:slow_fast_system}.

\begin{lemma}\label{lemma:any-beta}
For any $\varkappa\in(-1,1)$ and any $\beta\in B_{\varkappa}$ 
there exist admissible $\lambda^+$ and $\lambda^-$ such that $\beta$ satisfies \eqref{eq:def-of-beta}.
\end{lemma}

\begin{proof}[Proof of \autoref{lemma:any-beta}]
Taking the derivate of \eqref{eq:def-of-beta}, one obtain:
\begin{equation}\label{eq:beta-lambda}
	\lambda^-(y) = \beta'(y)\lambda^+(\beta(y)).
\end{equation}
Check that for any function $\lambda^+$ that satisfies \eqref{eq:l-2} the corresponding function $\lambda^-$ defined by \eqref{eq:beta-lambda} also satisfies \eqref{eq:l-2}. Recall that due to \eqref{eq:beta-construction}
$\beta'(y) = \sqrt{\theta(y)}$ for some $\theta\in\Theta$.

Consider some small right half-neighbourhood  $V=[-1, \delta_2)$. For any $y\in V$ 
\[
    \lambda^-(y)=-\sqrt{\theta(y)}\lambda^+(\beta(y)),~\beta(-1)=\varkappa>-1.
\]
The last factor in the right-hand side tends to a positive constant as $y\searrow -1$, therefore $\lambda^{-}$ satisfies condition \eqref{eq:l-2} near point $-1$.

Consider small left half-neighbourhood $V=(1-\delta_3, 1]$. For any $y\in V$
\[
    \lambda^-(y) = \beta'(y)\sqrt{\theta^+(y)},\quad \beta'(1)\ne 0.
\]
It follows that $\lambda^-$ also satisfies \eqref{eq:l-2} near point $1$.

Now satisfy condition \eqref{eq:l-1}. Consider the sequence of implications: 
\begin{enumerate}
\item The jet of $\lambda^+$ at $y=\varkappa$ defines (due to \eqref{eq:beta-lambda}) the jet of $\lambda^-$ at $y=-1$.
\item The jet of $\lambda^-$ at $y=-1$ defines (due to \eqref{eq:l-1}) the jet of $\lambda^+$ at $y=-1$.
\item The jet of $\lambda^+$ at $y=-1$ defines (due to \eqref{eq:beta-lambda}) the jet of $\lambda^-$ at $y=1$.
\item The jet of $\lambda^-$ at $y=1$ defines (due to \eqref{eq:l-1}) the jet of $\lambda^+$ at $y=1$.
\end{enumerate}
This sequence imposes some conditions on the derivatives of $\lambda^+$ at the points $y=-1,\varkappa,1$. For any function $\lambda^+$ that satisfies \eqref{eq:l-1}, \eqref{eq:l-2} and these conditions the corresponding function $\lambda^-$ will satisfy \eqref{eq:l-1}.

This proves the Lemma.
\end{proof}

The only thing to do is to prove \autoref{prop:Q-decomp} and \autoref{prop:Q-decomp-deriv}.

\section{Proof of approximations of the Poincar\'e map}\label{sec:techn}
\subsection{$C^0$-approximation}\label{ssec:approx}
In this section we provide a rigorous proof of \autoref{prop:Q-decomp}. See the
statement and the heuristic proof in \autoref{ssec:heur}.

\begin{proof}
    The proof goes in 3 steps.

\step{No too early releases} 
First we show that the trajectory $\psi$ with the initial condition $\psi(y_0)=0$ cannot dettach from the grand canard
at some fixed point to the left of $\beta(y_0)$. 

Indeed, consider arbitrary point
$\beta_1\in (\wt\alpha^-,\beta(y_0))$. 
Due to the definition of the map $\beta$ (see \eqref{eq:def-of-beta}) and the monotonicity of 
$\int_{y_1}^{y_2}\lambda^+dy$
 with respect to~$y_2$, the following estimate holds:
\begin{equation}\label{eq:beta_1}
\int_{y_0}^{1}\lambda^-dy + 
\int_{-1}^{\beta_1}\lambda^+dy
< 0
\end{equation}
Consider a vertical segment $J_0=\{y_0\}\times(x_1,x_2)$ that intersects $M^-$ and $\Sigma^-$  but does not intersect $M^+$. The trajectory $\psi$ intersects $J_0$.
Let the segments of the grand canard in the basic strip be given by functions $\xgc_+(y)$
(for the part near~$M^+$) and $\xgc_-(y)$ (for the part near $M^-$). 
It follows from
the geometric singular perturbation theory~\cite{Fe} that outside of some
neighbourhood of the jump points 
$$\xgc_\pm(y)=M^\pm(y)+O(\eps).$$  
Therefore, 
the
grand canard intersects $J_0$ at the point $\xgc_-(y_0)<0$. Let 
$$J_0'=\rPeab{y_0}{\pi}([\xgc_-(y_0),0]).$$
It follows from~\autoref{thm:P'est-full} that 
\begin{equation}\label{eq:J_0}
|J_0'|\le 
C_1
\exp
\dfrac{1}{\varepsilon}
\left(
\int_{y_0}^{\pi}f'_xdy + o(1)
\right)
\end{equation}
for some constant $C_1$.

Now consider arbitrary fixed vertical segment
$J_1=\{\beta_1\}\times(x_1',x_2')$
that intersects $M^+$ and does not intersect~$M^-$. As previously, the grand
canard intersects $J_1$ at some point $\xgc_+(\beta_1)$ for $\eps$ small enough. We will prove that the
trajectory $\psi$ intersects~$J_1$ as well. Let 
$$J_1'=\lPeab{\beta_1}{-\pi}([\xgc_+(\beta_1),x_2']).$$
It follows
from~\autoref{thm:P'est-full} that
\begin{equation}\label{eq:J_1}
   |J_1'|\ge C_2 \exp
   \dfrac{1}{\varepsilon}
   \left(
   		\int_{-1}^{\beta_1}f'_x dy + o(1)
	\right).
\end{equation}
The lower borders of the segments $J_0'$ and $J_1'$ coincide: both are the point
of the intersection of the grand canard and the cross-section $\{y=\pm \pi\}$.
It follows from~\eqref{eq:beta_1}, \eqref{eq:J_0} and~\eqref{eq:J_1} that
$|J_1'|\gg |J_0'|$ and therefore $J_0'\subset J_1'$. As trajectory $\psi$
intersects $J_0'$ it therefore intersects $J_1'$ and $J_1$.

\step{No too late releases}
Prove that $Q^-_\eps(y_0)$ exists and moreover $Q^-_\eps(y_0)<\beta_2$ for any  fixed $\beta_2\in(\beta(y_0),\wt \alpha^+)$ and sufficiently small $\varepsilon$. On the previous step we proved that the trajectory $\psi$ intersects the vertical segment $J_2$ to the left of $\beta(y_0)$. Now consider a segment $J_3=\{\beta_2\}\times(0, \pi)$ to the right of $\beta(y_0)$. Then the trajectory $\psi$ does not intersect $J_3$. This can be proved by application of the previous step to the system with the time reversed.

Now consider a region bounded by the circles $y=\beta_1$, $y=\beta_2$, the grand canard $x=\xgc_+(y)$ and
the segment $\Sigma^+$. The trajectory $\psi$ enters this region at some point of the segment $J_1$. As $y$ monotonically increase, the trajectory has to leave the region. It cannot intersect $y=\beta_2$ and the grand canard. Therefore, it intersects $\Sigma^+$ at some point $(Q_{\varepsilon}^+(y_0),\pi)$. 

\step{Asymptotics of $Q_{\varepsilon}^-(y_0)$}
Assume that $Q_{\varepsilon}^-(y_0)$ is defined. Let
$$R^2_\eps=\rPeab{-\pi}{Q_{\varepsilon}^-(y_0)}\circ\rPeab{y_0}{\pi}.$$
Then
$$\frac{R^2_\eps(0)-R^2_\eps(\xgc_-(y_0)}{0-\xgc_-(y_0)}=\frac{\pi-M^+(y_0)+O(\eps)}{0-M^-(y_0)+O(\eps)}.$$
Due to Mean Value Theorem, it follows that the derivative of $R_\eps$ at some point~$x^*$ has to be of order constant. However Theorem~\ref{thm:P'est-full} and the chain rule imply that
\begin{equation}\label{eq:Reps2}
    \log (R^2_\eps)'_x=
    \dfrac{1}{\varepsilon}
    \left(
    	\int_{y_0}^{1}\lambda^+dy +
    	\int_{-1}^{Q_{\varepsilon}^{-}(y_0)}\lambda^-dy +	
		O(\varepsilon^{\nu})    
    \right).
\end{equation}
The right-hand side of~\eqref{eq:Reps2} is bounded away from zero and infinity if the following estimate holds: 
$$
\int_{y_0}^{1}\lambda^+dy 
+
\int_{-1}^{Q_{\varepsilon}^{-}(y_0)}\lambda^-dy 	
=
O(\eps^\nu),
$$
which differs from equation~\eqref{eq:def-of-beta} only on $O(\eps^\nu)$. Inverse Function Theorem finishes the proof.\end{proof}
\subsection{$C^1$-approximation}\label{ssec:Q-deriv}
In this section we prove \autoref{prop:Q-decomp-deriv}.

Denote the coordinates on $\Sigma^{\pm}$ by $y_{\pm}$. 
Decompose the Poincar\'e map $Q_{\varepsilon}$ in the following way: 
\begin{equation}\label{eq:Q-decomp2}
Q_{\varepsilon}\colon 
y_{-}^{(1)}\stackrel{Q^-_\eps}{\mapsto}y_{+}^{(2)}\stackrel{Q^+_\eps}{\mapsto} y_{-}^{(3)}.
\end{equation}
Consider a trajectory that intersects $\Sigma^{-}$ at the point $y^{(1)}_{-}$, after that $\Sigma^{+}$ at the point $y^{(2)}_{+}$, after that $\Sigma^{-}$ at the point $y^{(3)}_{-}$.
For any $y$ definition \eqref{eq:def-of-beta} of the map $\beta$ implies
\begin{equation}\nonumber
	\beta'(y) = \dfrac{\lambda^-(y)}{\lambda^+(\beta(y))}.
\end{equation}
By the definition $\lambda^-(y) = f'_{x} (M^-(y),y) = f'_x (x^{gc}_{-}(y),y) + o(1)$. Due to \autoref{prop:Q-decomp} $y_+^{(2)} = Q_{\varepsilon}^{-} (y_{-}^{(1)}) = \beta(y_{-}^{(1)}) + O(\varepsilon^{\nu})$.  
Therefore the statement of \autoref{prop:Q-decomp-deriv} follows from the  following estimates:
\begin{equation}\label{eq:Q^--deriv-full}
		(Q_{\varepsilon}^{-})' 
		=
		\dfrac  {f'_x(x^{gc}_{-}(y_{-}^{(1)}),y_{-}^{(1)})}
		        {f'_x(x^{gc}_{+}(y_{+}^{(2)}),y_{+}^{(2)})}
		\cdot (1+o(1)),
\end{equation}
\begin{equation}\nonumber
		(Q_{\varepsilon}^{+})'
		=
		\dfrac  {f'_x(x^{gc}_{-}(y_{+}^{(2)}),y_{+}^{(2)})}
		        {f'_x(x^{gc}_{+}(y_{-}^{(3)}),y_{-}^{(3)})}
		\cdot (1+o(1)).
\end{equation}

Introduce some notation.
Fix parameter $\varepsilon\in R_n$ such that the grand canard exists. 
Fix the corresponding grand canard. Consider its part which is bounded by the intersections with $J^+$ and $J^-$. It can be described as double-valued function $y\mapsto x$.   Denote it by $x=x^{gc}(y)$. We will also use the notation $x=x^{gc}_+(y)$ and $x=x^{gc}_-(y)$ in the neighborhood of $M^+$ and $M^-$ respectively. 

Recall Theorem 3 from \cite{GI} and equation (3.5) from  \cite{IS2}:

\begin{theorem} 
\label{thm:linearization}
Family \eqref{eq:slow_fast_system} for small $\varepsilon>0$ in the neighborhood of the slow curve (and outside any small fixed neighborhood of the jump points) is smoothly orbitally equivalent to the family
$$
	\dot{x}=f'_x(s(y,\varepsilon),y,\varepsilon),~
	\dot{y}=\varepsilon.
$$
\end{theorem}
Due to \cite{IS2} the function $x=s(y,\varepsilon)$ above represents the true slow curve; it is defined in a non-unique way but all true slow curves are exponentially close to each other and one can pick arbitrary one.
We choose segments of the grand canard bounded away from the jump points as this true slow curve, so put $s(y,\varepsilon)=x^{gc}_{\pm}(y,\varepsilon)$. 
Below we will omit $\varepsilon$ in $x^{gc}_{\pm}(y,\varepsilon)$ for brevity.
In \autoref{thm:linearization} take the neighborhood of the jump points sufficiently small such that small neighbourhood of $\{(M^{\pm}(y),y)\mid y\in I_\delta\}$ is strictly inside the linearized area (we finally fix the linearized area and the basic strip after the proof of  \autoref{prop:alpha-exists}). 
Consider two vertical transversal segments $\hat{J}^{\pm}$ that intersect  $M^+$ and $M^-$ respectively in the neighbourhood of the jump points. Denote their $y$-coordinates by $\hat{\alpha}^{+}$ and $\hat{\alpha}^{-}$, $\hat{\alpha}^{\pm} = \hat{\alpha}^{\pm}(\varepsilon)$.

\begin{proposition} 
\label{prop:alpha-exists}
One can take parameters 
$\hat{\alpha}^{\pm}(\varepsilon)$ depending on $\eps$ and
bounded away from $y(G^{\pm})$ uniformly in $\eps$
such that the following equality takes place for sufficiently small values of $\varepsilon$:
\begin{equation}\label{eq:int-zero-gc}
	\int_{\hat{\alpha}^-}^{\hat{\alpha}^+}
        f'_x
        (x^{gc}(y),y,\varepsilon)
    dy 
    = 0.
\end{equation}
\end{proposition}

\begin{proof}[Proof of \autoref{prop:alpha-exists}]
Equation in variations implies:
\begin{equation}\label{eq:hell-proposition-0}
	\int_{\hat{\alpha}^-}^{\hat{\alpha}^+}
        f'_x
        (x^{gc}(y),y,\varepsilon)
    dy 
    =
    \varepsilon
    \cdot
    \log
    	\left(
    		P	_{\varepsilon}
    			^{\rarc{\hat{\alpha}^-}{\hat{\alpha}^+}}
			(x^{gc})
    	\right)'_x.
\end{equation}
Due to \autoref{thm:P'est-full} for some 
$r(\varepsilon) = O(\varepsilon^{\nu})$ the following holds
\begin{equation}\label{eq:hell-proposition-begins}
\varepsilon
   \cdot
   \log
   	\left(
   		P	_{\varepsilon}
   			^{\rarc{\hat{\alpha}^-}{\hat{\alpha}^+}}
		(x^{gc})
   	\right)'_x
=
\int_{\hat{\alpha}^-}^{1} f'_x(x^{gc}_-) dy
+
\int_{-1}^{\hat{\alpha}^+} f'_x(x^{gc}_+) dy
+
r(\varepsilon).
\end{equation}
Consider sum of the integrals in the right-hand side of the estimate. The functions 
$\int_{A}^{B} f'_x(x^{gc}_{\pm}) dy$
 are continuous and monotonic in $A$, $B$ and tend to zero as $B$ tends to $A$. Fix any point $\alpha_1^+>-1$, $\alpha_1^+$ close to $-1$. Then for any $\alpha_1^-<1$ such that $\alpha_1^-$ sufficiently close to $1$ the corresponding sum of the integrals is positive
\begin{equation}\label{eq:positive}
\int_{\alpha_1^-}^{1} f'_x(x^{gc}_-) dy
+
\int_{-1}^{\alpha_1^+} f'_x(x^{gc}_+) dy
>0.
\end{equation}
Define in the same way $\alpha_2^-<\alpha_1^-$ and 
$\alpha_2^+\in(-1,\alpha_1^+)$ such that sum of integrals is negative:
\begin{equation}\label{eq:negative}
\int_{\alpha_2^-}^{1} f'_x(x^{gc}_-) dy
+
\int_{-1}^{\alpha_2^+} f'_x(x^{gc}_+) dy
<0
\end{equation}
Let constant $\hat{c}>0$ is small enough such that zero in the right-hand side of  inequalities~\eqref{eq:positive} and~\eqref{eq:negative} can be replaced by $\hat{c}$ and $-\hat{c}$ respectively. For sufficiently small $\varepsilon$ we can estimate the last term in \eqref{eq:hell-proposition-begins}:
\begin{equation}\nonumber
	\hat{c}>|r(\varepsilon)|.
\end{equation}
Summarize equalities \eqref{eq:hell-proposition-0}, \eqref{eq:hell-proposition-begins} and estimates given above:
\begin{equation}
	\int_{\alpha_1^-}^{\alpha_1^+}
        f'_x
        (x^{gc}(y),y,\varepsilon)
    dy > 0, 
\end{equation}
\begin{equation}
	\int_{\alpha_2^-}^{\alpha_2^+}
        f'_x
        (x^{gc}(y),y,\varepsilon)
    dy < 0.
\end{equation}
Consider $\tau\in[0,1]$ and $\hat{\alpha}^{\pm}(\tau) := \tau \alpha_1^{\pm} + (1-\tau) \alpha_2^\pm$. By continuity there exists $\tau_{\varepsilon} = \tau(\varepsilon)$ such that \eqref{eq:int-zero-gc} holds for $\hat{\alpha}^{\pm}(\tau_{\varepsilon})$.
\end{proof}
Now take $\delta$ in $I_\delta$ smaller than $\min (dist(1,\alpha_1^-), dist(\alpha_2^+,-1))$ and linearized area in \autoref{thm:linearization} large enough such that 
$\hat{\alpha}^{\pm}\in I_\delta$ and the segments $\hat{J}^{\pm}$ are in the linearized area.

Denote by $\eta_{\pm}$ the vertical component of the linearizing charts in the neighborhood of $M^{\pm}$ respectively. 
Moreover, due to the choice of $s(y,\varepsilon)$ in \autoref{thm:linearization}, the grand canard $x^{gc}(y)$ is given by the equation $\eta_{\pm}=0$ in the neighborhood of $M^{\pm}$.
Take small constant $c$ and let $\xi_{\pm}$ be $y$-coordinate on the line $\{\eta_{\pm}=c\}$. 
Decompose $Q_{\varepsilon}^-$ as
$$
	Q_{\varepsilon}^-\colon
	\Sigma^-\to \{\eta_-=c\}\to \hat{J}^- \to \hat{J}^+ 
	\to \{\eta_+=c\}\ \to 	\Sigma^+,
$$
$$
	Q_{\varepsilon}^-\colon
	y_{-}\mapsto\xi_{-}\mapsto\eta_{-}\mapsto
	\eta_{+}\mapsto\xi_{+}\mapsto y_{+}.
$$ 
\begin{figure}[t!]
\centering
\includegraphics{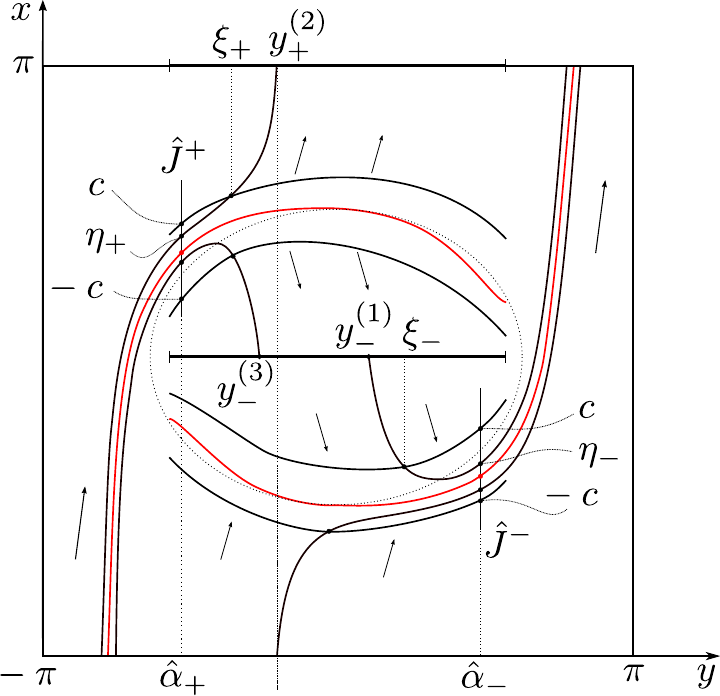}
\caption{Coordinates in \autoref{prop:Q-decomp-deriv}}
\end{figure}

\begin{proof}[Proof of \autoref{prop:Q-decomp-deriv}]
Below we prove only equality \eqref{eq:Q^--deriv-full} for the map 
$Q_{\varepsilon}^{-}$. Proof for $Q_{\varepsilon}^{+}$ is similar.

We have to estimate the derivative
\begin{equation} \label{eq:Q-'}
	\left({Q_{\varepsilon}^{-}}\right)'=
	\dfrac{dy_{+}}{dy_{-}} = 
    \dfrac{dy_{+}}{d\xi_{+}}\cdot
	\dfrac{d\xi_{+}}{d\eta_{+}}\cdot
	\dfrac{d\eta_{+}}{d\eta_{-}}\cdot
	\dfrac{d\eta_{-}}{d\xi_{-}} \cdot
    \dfrac{d\xi_{-}}{dy_{-}}.
\end{equation}
The constant $c$ in the definition of the cross-section $\{\eta_{\pm}=c\}$ does not depend on $\varepsilon$ and the trajectory spends bounded time between the intersections with $\Sigma^{\pm}$ and $\{\eta_{\pm}=c\}$, therefore
\begin{equation}\label{eq:xi-y}
	   \xi_{\pm}(y_{\pm}) = y_{\pm} + \varepsilon K_{\pm}(\varepsilon,y_{\pm}),~
	            \text{$K_{\pm}$ is smooth.}
\end{equation}

The charts $(\eta_{\pm},\xi_{\pm})$ are linearized, therefore due to \autoref{thm:linearization} the following takes place:
\begin{equation}\nonumber
	\eta_{-}(\xi_{-}) 
		= 
		c
		\exp
		\frac{1}{\varepsilon}
		\int_{\xi_{-}}^{\hat{\alpha}^{-}}f'_x(x^{gc}_{-}(y),y)dy,	
\end{equation}
\begin{equation}\nonumber
\eta_{+}(\xi_{+}) 
	= 
		c
		\exp
		\left(-
		\frac{1}{\varepsilon}
		\int_{\hat{\alpha}^{+}}^{\xi_{+}}f'_x(x^{gc}_{+}(y),y)dy
		\right).
\end{equation}
Calculate the derivatives:
\begin{equation}
\label{eq:eta_1}
	\begin{aligned}
		\frac{d\eta_{-}}{d\xi_{-}} & = 
			-\frac{c}{\varepsilon}
			f'^{-}_x(x^{gc}_{-}(\xi_{-}),\xi_{-})
			\exp
		\frac{1}{\varepsilon}
		\int_{\xi_{-}}^{\hat{\alpha}^{-}}f'_x(x^{gc}_{-}(y),y)dy, \\
		\frac{d\eta_{+}}{d\xi_{+}} & = 
			-\frac{c}{\varepsilon}
			f'^{+}_x(x^{gc}_{+}(\xi_{+}),\xi_{+})
			\exp
		\left(-
		\frac{1}{\varepsilon}
		\int_{\hat{\alpha}^{+}}^{\xi_{+}}f'_x(x^{gc}_{+}(y),y)dy
		\right) . \\
	\end{aligned}
\end{equation}
Below we omit the argument of $x^{gc}_{\pm}$ for brevity.

Now estimate the factors in the decomposition \eqref{eq:Q-'} for some trajectory $x(y)$ using \eqref{eq:xi-y} and \eqref{eq:eta_1} :

\begin{multline}\label{eq:Q'decomp}
	{Q_{\varepsilon}^{-}}' = 
    \dfrac{1+\varepsilon K_{-}}{1+\varepsilon K_{+}}\cdot
	\dfrac	{f'^{-}_x(x^{gc}_{-},\xi_{-})}
			{f'^{+}_x(x^{gc}_{+},\xi_{+})}\cdot
	\exp{
		\dfrac{1}{\varepsilon}
		\left(
			\int_{\xi_{-}}^{\hat{\alpha}^{-}}f'_x(x^{gc}_{-},y)dy +
			\int_{\hat{\alpha}^{+}}^{\xi_{+}}f'_x(x^{gc}_{+},y)dy
		\right)
	}\cdot
	\dfrac{d\eta_{+}}{d\eta_{-}}
\end{multline}
Estimate the last factor:
\begin{equation}
\label{eq:eta_factor}
    \left.
    \dfrac
        {d\eta_{+}}
        {d\eta_{-}}
    \right|_{(x(y),y)}
    =
    \exp{
        \frac
            {1}
            {\varepsilon}
        \int_{\hat{\alpha}^-}^{\hat{\alpha}^+}
            f'_x(x(y),y,\varepsilon)
        dy
    }.
\end{equation}
Consider the integral in the argument of $\exp$-function.
By construction the trajectory $x(y)$ intersects $\{\eta_-=c\}$ at the point $\xi_-$ inside  segment $[a_0,b_0]$, and $\xi_-<\hat{\alpha}^-$.
Due to \autoref{thm:P'est-full}  for any $y\in\rarc{1}{-1}$ 
the following approximations take place:
\begin{equation}\label{eq:693}
	\left(
		P^{\rarc{\xi_-}{y}}_{\varepsilon}
	\right)_x'
	=
	\exp
		\frac{1}{\varepsilon}
		\left(
			\int_{\xi_{-}}^{1}\lambda^{-}dy
			+
			O(\varepsilon^{\nu})
		\right),
\end{equation}
$$
	\left(
		P^{\rarc{y}{\hat{\alpha}^+}}_{\varepsilon}
	\right)_x'
	=
	\exp
		\frac{1}{\varepsilon}
		\left(
			\int_{-1}^{\hat{\alpha}^+}\lambda^+dy
			+
			O(\varepsilon^{\nu})
		\right).
$$
Due to the definition of $\hat{\alpha}^{\pm}$ (see \autoref{prop:alpha-exists}) the following holds: there exists a constant $M>0$ such that
\begin{equation}\label{eq:694}
	\int
		_{\xi_{-}}
		^{\hat{\alpha}^-}
		\lambda^-
	dy 
	+
	\int
		_{\hat{\alpha}^-}
		^{1}
		\lambda^-
	dy
	+
	\int
		_{-1}
		^{\hat{\alpha}^+}
		\lambda^+
	dy
	<
	-M
    .
\end{equation}
Therefore for any $y\in(\hat{\alpha}^-,\hat{\alpha}^+)$ the following estimate takes place:
$$
	x(y)-x^{gc}(y) <
	\exp\left(-\frac{M}{\varepsilon}\right),\quad
    M=\mathrm{const}.
$$
Indeed, for $y\in(\xi_-,1)$ it follows from \eqref{eq:693} and for $y\in(1,\alpha^+)$ it follows from \eqref{eq:694}.
Further,
\begin{equation}\label{eq:o-eps-compare}
  \int_{\hat{\alpha}^-}^{\hat{\alpha}^+}
  \left(
    f'_x(x(y),y,\varepsilon)-
    f'_x(x^{gc}(y),y,\varepsilon)
  \right)
  dy 
  <
	\int_{\hat{\alpha}^-}^{\hat{\alpha}^+}
    \int_0^{\exp\left({-{M}/{\varepsilon}}\right)}
      f''_{xx} 
    dx 
  dy 
  =
  o(\varepsilon).
\end{equation}
Due to \autoref{prop:alpha-exists} the left-hand side of inequality \eqref{eq:o-eps-compare} equals 
$
\int_{\hat{\alpha}^-}^{\hat{\alpha}^+}
f'_x(x(y),y,\varepsilon) 
dy
$.
Therefore due to \eqref{eq:eta_factor} the  
following estimate takes place:
$$
	\left.
    \dfrac
        {d\eta_{+}}
        {d\eta_{-}}
    \right|_{(x(y),y)}
    =
    1+o(1).
$$
Now estimate the factor with $\exp$-function in \eqref{eq:Q'decomp}:
\begin{equation}\label{eq:thirdFactor}
	\exp{
		\dfrac{1}{\varepsilon}
		\left(
			\int_{\xi_{-}}^{\hat{\alpha}^{-}}f'_x(x^{gc}_{-},y)dy +
			\int_{\hat{\alpha}^{+}}^{\xi_{+}}f'_x(x^{gc}_{+},y)dy
		\right)
	}
\end{equation} We use the following obvious statement:

\begin{proposition} 
\label{prop:x*}
  Let for some $\xi_{\pm}$ there exists a trajectory $x(y)$ that passes through the points $(\xi_{-},c)$ and $(\xi_{+},c)$. Then there exists a trajectory $x^*(y)$ such that
  \begin{equation}\label{eq:int-zero-*}
    \int_{\xi_{-}}^{\xi_{+}}
    f'_x
    (x^*(y),y,\varepsilon)
    dy 
    = 0.
  \end{equation} 
\end{proposition}
\begin{proof}[Proof of \autoref{prop:x*}]
  The grand canard passes through  $(\xi_{-},0)$ and $(\xi_{+},0)$. 
  Therefore by Mean Value Theorem for the Poincar\'e map from the segment $[(\xi_{-},0),(\xi_{-},c)]$ to the segment $[(\xi_{+},0),(\xi_{+},c)]$ there exists a trajectory $x^*(y)$ with the derivative of the Poincar\'e map that is equal to $1$. Equation in variations implies \eqref{eq:int-zero-*}.
\end{proof}

Let us decompose \eqref{eq:int-zero-*} in the sum of three integrals:
$$
	\int_{\xi_{-}}^{\xi_{+}} f'_x (x^*(y),y,\varepsilon) dy =
	\int_{\xi_{-}}^{\hat{\alpha}^-} +
	\int_{\hat{\alpha}^-}^{\hat{\alpha}^+}+
	 \int_{\hat{\alpha}^+}^{\xi_{+}} =0.
$$
Literally as in \eqref{eq:o-eps-compare} the second integral has an order $o(\varepsilon)$.
Charts in the neighbourhood of the slow curve are linearized so one can replace $x^*(y)$ by
$x^{gc}(y)$ in the first and the third integrals. So the following takes place
$$
\int\limits_{\xi_{-}}^{\hat{\alpha}^{-}}f'^-_x(x^{gc}_{-},y)dy + 
o(\varepsilon)+
\int\limits_{\hat{\alpha}^{+}}^{\xi_{+}}f'^+_x(x^{gc}_{+},y)dy = 0,
$$
therefore the sum of the integrals in \eqref{eq:thirdFactor} has an order of $o(\varepsilon)$.
Therefore the third factor has an order $\exp({o(\varepsilon)}/{\varepsilon})=1+o(1)$.
The derivative $f'_{x}$ is smooth therefore due to \eqref{eq:xi-y} the following takes place:
$f'_{x}(x^{gc}_{\pm}(y_{\pm}), y_{\pm})=
 f'_{x}(x^{gc}_{\pm}(\xi_{\pm}),\xi_{\pm}) + O(\varepsilon)$.
Finally
$$
\dfrac{dy_{+}}{dy_{-}} 
    =
    \dfrac  {f'_x(x^{gc}_{-},y_{-})}
            {f'_x(x^{gc}_{+},y_{+})}
    \cdot (1+o(1))
    \cdot (1+O(\varepsilon)).
$$
\end{proof}

\section{Conjectures on unconnected slow curves}\label{sec:conjectures}
Previously, only the case of connected slow curve was studied. We believe that the
methods of the present paper can be applied also to the case of unconnected
slow curve. In this section we propose a related conjecture and outline its
 proof.

Consider system \eqref{eq:slow_fast_system} on the two-torus with the slow curve $M$ that satisfies the following properties:
\begin{enumerate}
	\item The slow curve $M$ has two connected components $M_1$ and $M_2$. Projection of $M$ on the $y$-axis along the $x$-axis is two disjoint arcs. 
	\item Each of the components $M_1$, $M_2$ has two jump points $G_1^{\pm}$ and $G_2^{\pm}$.
	\item The components $M_1$, $M_2$ are nondegenerate in the sense
        of~\autoref{def:nondeg}.
\end{enumerate}
We will call such slow curve \emph{simple unconnected}.

\begin{conjecture}\label{conj1}
For every desired number $l\in\mathbb{N}$ there exists a local topologically generic set in the space of slow-fast systems on the two-torus with the following property:
\begin{enumerate}
	\item Slow curve is nondegenerate and simple unconnected.
	\item For every system from this set there exists a sequence of intervals
			$\{R_n\}_{n=0}^{\infty}\subset\{\varepsilon>0\}$,
			accumulating at zero,
			such that for every $\varepsilon\in R_n$
			there exists at least $l$ canard limit cycles that makes one pass along the slow direction.
\end{enumerate}
\end{conjecture}

\begin{figure}[htb]
    \begin{center}
        \includegraphics[width=10cm]{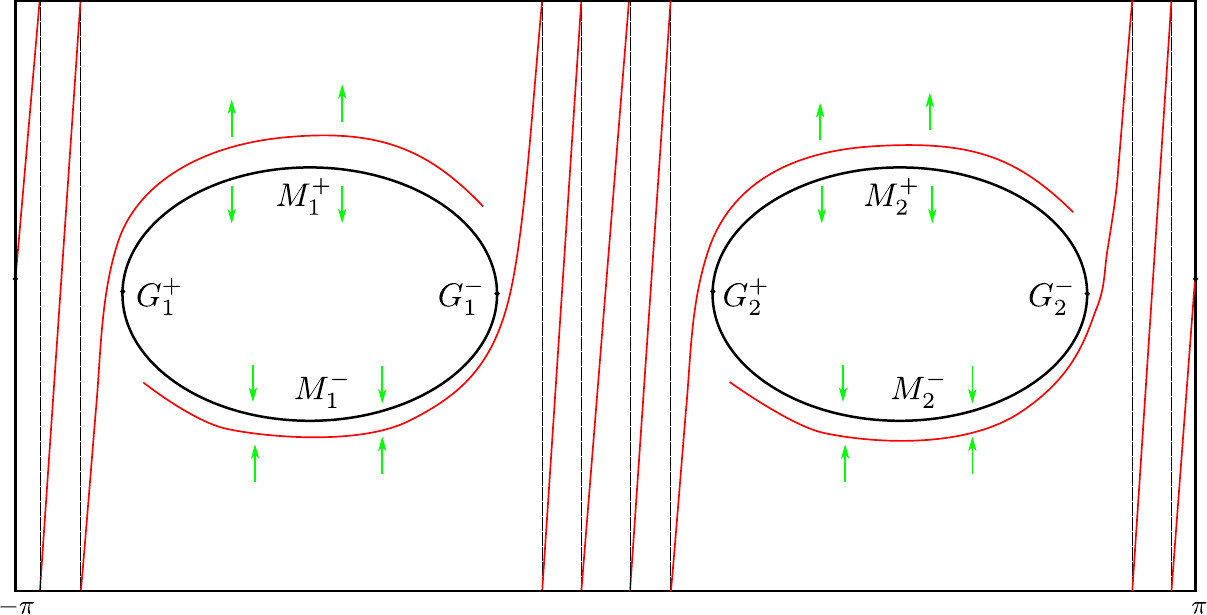}
        {\vskip 0.3cm}
        \includegraphics[width=10cm]{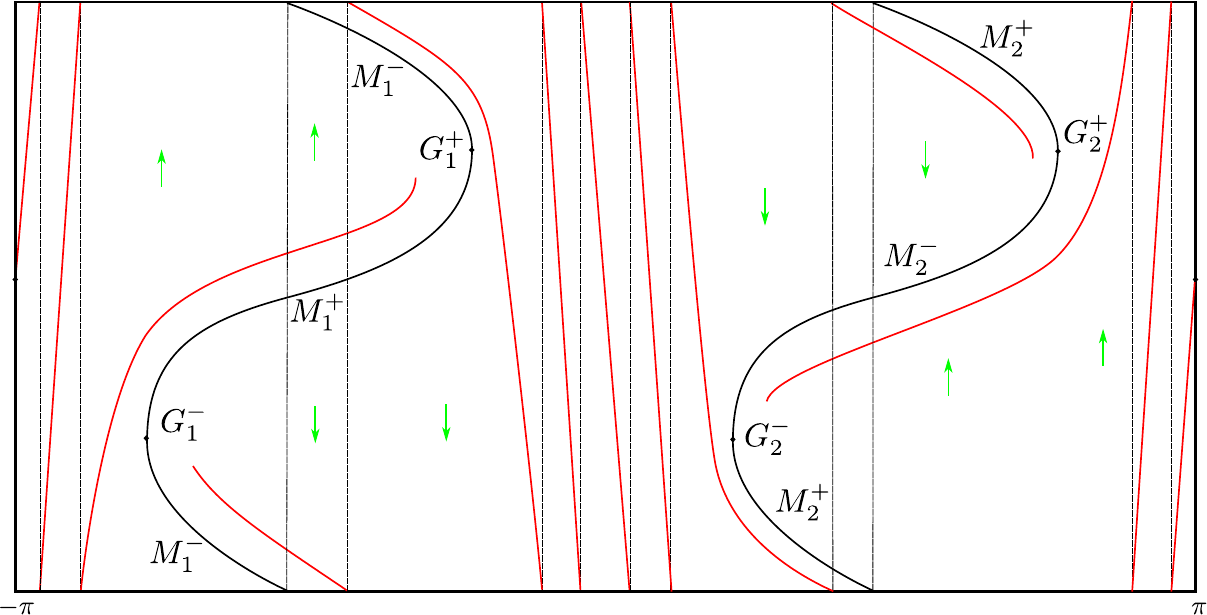}
        \caption{Simple unconnected slow curve: contractible (top) and noncontractivle (bottom) cases}\label{fig:unconnected_slow_curve}
    \end{center}
\end{figure}%

The components of the simple unconnected slow curve may be both contractible or both noncontractible, see \autoref{fig:unconnected_slow_curve}. Consider first the case of contractible ones, see \autoref{fig:unconnected_slow_curve}, top. Note that this figure is very similar to one studied in the present paper (see \autoref{fig:general-view}) but contains two copies of the slow curve. In fact one can obtain such a picture as a two-leaf cover of the original system~\eqref{eq:slow_fast_system}: the fundamental domain should be extended twice along the slow direction. If we obtain the new system from the original one in this way all the results about two-pass limit cycles for the original system can be rewritten in terms of one-pass limit cycles for the new system. This is not a huge success because of strong nongenericity of the new system (shift symmetry). However, it is easy to see that the only requirement that is needed for our proof to work in the new settings is the existence of two grand canards for the same value of $\eps$. For shift-symmetric system we obtain them ``for free'', for generic system it seems to be much rarer event.  Nevertheless, if the generic system possesses two grand canards like it is shown on the Figure, the same arguments we used to prove \autoref{thm:main} can be applied verbatim to the proof of \autoref{conj1}.

We also note that the same arguments work for the case of noncontractible components, again, provided that two grand canards exists. This
can be of particular interest because slow-fast systems with 
noncontractible slow curve (like in \autoref{fig:unconnected_slow_curve}, bottom), 
appear naturally in physical applications~\cite{KRS}.

Thus the key question is: can two grand canards coexist for the same value of
$\eps$ for generic system? We conjecture that for \emph{topologically generic}
systems the answer is affirmative.
\begin{conjecture}
    For \emph{locally topologically generic} slow-fast system on the two-torus there
    exists a sequence of intervals on the ray $\{\eps >0 \}$ accumulating at 0
    such that for every $\eps$ from these intervals there exist two grand
    canards.
\end{conjecture}
\begin{proof}[Outline of the proof]
Each component $M_{1}$, $M_{2}$ has repelling and attracting parts. Denote them by $M_{1}^{\pm}$ and $M_{2}^{\pm}$ respectively.
Take $\alpha_1^{\pm}$ close to $y(G_1^{\pm})$, $\alpha_1^{\pm}\in \left(y(G_1^+),y(G_1^-)\right)$ and  
$\alpha_2^{\pm}$ close to $y(G_2^{\pm})$, $\alpha_2^{\pm}\in \left(y(G_2^+),y(G_2^-)\right)$. 
As before (see \autoref{section:notation}) consider corresponding vertical transversal segments $J_{1}^{\pm}$, $J_{2}^{\pm}$ with the following properties: for $i=1,2$ segment $J_i^{\pm}$ intersects $M_i^{\pm}$ respectively and does not intersect $M_i^{\mp}$; moreover,
$$
	y(J_i^{\pm}) = \alpha_i^{\pm},~i=1,2.
$$
Denote $y_{1} = (y(G_1^-) + y(G_2^+))/2$, 
$y_{2} = (y(G_2^-) + y(G_1^+))/2$ and consider
 vertical global transversal circles 
$y_1 \times \mathbb{S}^1$ and 
$y_2 \times \mathbb{S}^1$.
As before (see \autoref{section:preliminary}) denote
$$
	D_{1,\varepsilon}^+ = P_{\varepsilon}^{\rarc{\alpha_1^+}{y_1}} J_1^+,~
	D_{1,\varepsilon}^- = P_{\varepsilon}^{\larc{y_2}{\alpha_1^-}} J_1^-,	
$$
$$
	D_{2,\varepsilon}^+ = P_{\varepsilon}^{\rarc{\alpha_2^+}{y_2}} J_2^+,~
	D_{2,\varepsilon}^- = P_{\varepsilon}^{\larc{y_1}{\alpha_2^-}} J_2^-.
$$
If intersection $D_{1,\varepsilon}^+\cap D_{2,\varepsilon}^-$ is not empty, then there exists a grand canard that intersects both $J_1^+$ and $J_2^-$. The same holds for the intersection $D_{1,\varepsilon}^-\cap D_{2,\varepsilon}^+$ and a grand canard that intersects both $J_1^-$ and $J_2^+$.

Further, these intersections forms a series of segments accumulating at zero.
$$
\mathcal R^1=\{R_n^1\} _{n=0}^{\infty},\quad \mathcal R^2=	\{R_n^2\}_{n=0}^{\infty}
$$
$$
\bigcup_{n=0}^\infty R_n^1=\{\varepsilon\mid D_{1,\varepsilon}^+\cap D_{2,\varepsilon}^- \ne \emptyset\},\quad
\bigcup_{n=0}^\infty R_n^2= \{\varepsilon\mid D_{1,\varepsilon}^-\cap D_{2,\varepsilon}^+ \ne \emptyset\}.
$$
\begin{definition}
System \eqref{eq:slow_fast_system} is called \emph{good} if for any $\varepsilon>0$ there exists $n,m\in \mathbb{N}$ such that
$R_n^1\cap R_m^2 \ne \emptyset$ and $R_n^1,~R_m^2\subset (0,\varepsilon)$. 
\end{definition}
We conjecture that \emph{good} systems are topologically generic in the space of slow-fast systems\footnote{The authors are grateful to Victor
        Kleptsyn for these arguments.}. Indeed, for a couple of intervals $R_k^1$ and $R_l^2$ one can perturb system slightly (increasing or decreasing the fast component of the vector field in the ``rotation phase'') to make them intersect each other by a segment. If $k$ and $l$ are large enough the perturbation is small enough. Therefore, the set of systems with one intersection between $\mathcal R^1$ and $\mathcal R^2$ is open and dense. In a similar way one can show that the set of systems with any finite number of intersections between these two series of segments is open and dense. It follows that the set of systems with \emph{infinitely many} intesections is residual and the corresponding property is topologically generic.
\end{proof}

\paragraph{Acknowledgements.} The authors are grateful to John Guckenheimer, Victor Kleptsyn, Alexey Klimenko and Alexey Okunev for fruitful discussions and valuable comments. The authors are especially grateful to Yulij Ilyashenko for his interest in the work and the suggestions on the text of the paper that drastically simplified the proof.

\end{document}